\documentclass[final,12pt]{msml2022} 

\title{Notes on Exact Boundary Values in Residual Minimisation}
\usepackage{times}
\usepackage{thm-restate}
\usepackage{enumitem}
\usepackage{todonotes}
\usepackage[normalem]{ulem}
\usepackage{amsmath}

\newtheorem{setting}[theorem]{Setting}

\usepackage{mathtools}

\newcommand{\jm}[1]{\textcolor{black}{#1}}
\definecolor{dkgold}{rgb}{0.639, 0.53, 0.19}
\newcommand{\mz}{\color{black}}
\newcommand{\eee}{\color{black}}

\msmlauthor{%
 \Name{Johannes M\"uller} \Email{jmueller@mis.mpg.de}\\
 \addr Max Planck Institute for Mathematics in the Sciences, Leipzig, Germany
 \AND
 \Name{Marius Zeinhofer} \Email{mariusz@simula.no}\\
 \addr Simula Research Laboratory, Oslo, Norway%
}

\makeatletter
 \let\Ginclude@graphics\@org@Ginclude@graphics
\makeatother

\begin{document}

\maketitle

\begin{abstract}
We analyse the difference in convergence mode using exact versus penalised boundary values for the residual minimisation of PDEs with neural network type ansatz functions, as is commonly done in the context of physics informed neural networks. It is known that using an $L^2$ boundary penalty leads to a loss of regularity of $3/2$ meaning that approximation in \(H^2\) yields a posteriori estimates in \(H^{1/2}\). These notes demonstrate how this loss of regularity can be circumvented if the functions in the ansatz class satisfy the boundary values exactly. Furthermore, it is shown that in this case, the loss function provides a consistent a posteriori error estimator in $H^2$ norm made by the residual minimisation method. We provide analogue results for linear time dependent problems and discuss the implications of measuring the residual in Sobolev norms.
\end{abstract}

\section{Introduction}
After their striking success in supervised learning tasks, neural network based methods have recently gained more attraction for problems from numerical analysis. Neural network based approaches for the approximate solutions of PDEs can be traced back to~\cite{dissanayake1994neural} where they proposed to combine residual minimisation with a boundary penalty term in order to train the parameters of a neural network. This ansatz was recently revived by \cite{sirignano2018dgm}, embracing increased computational power due to the heavy usage of GPUs and has received a growing amount of attention since then, see \cite{weinan2018deep, raissi2018hidden} and subsequent work.

Due to the unconstrained nature of neural networks, resolving boundary conditions for PDEs is challenging as noted by \cite{weinan2018deep, van2020optimally, cyr2020robust}. The work of \cite{lagaris1998artificial} proposes to encode boundary conditions directly into the trial functions. This approach was refined and extended by \cite{berg2018unified, lyu2020enforcing} and it was observed that such an ansatz frequently helps in producing more accurate solutions. 

As our main contribution we show that 
for residual minimisation $H^2$ approximation yields estimates in $H^2$ for ansatz classes with exact boundary values, see Theorem~\ref{thm:MainTheorem2}.
In contrast, 
using an $L^2$ penalty on the boundary values leads to error estimates in norms not stronger than $H^{1/2}$. 
The extension to more general equations, including time dependant ones, is discussed as well.

\subsection{Physics informed neural networks}\label{sec:pinns}
Suppose, we want to approximately solve the Dirichlet problem 
\begin{equation}\label{eq:Dir}
    \begin{split}
        -\Delta u & = f \quad \text{in } \Omega \\
    u & = g \quad \text{on } \partial\Omega,
    \end{split}
\end{equation}
where \(f\in L^2(\Omega)\) and \(g\in H^{3/2}(\partial\Omega)\). We assume that this problem is well posed and denote its solution by \(u_f\in H^2(\Omega)\)\mz, compare also to Remark~\ref{remark:regularity}. \eee There exist 
different approaches to this problem, which usually formulate the Dirichlet problem \eqref{eq:Dir} as a minimisation problem over a function space and then minimise this energy over a suitable ansatz class. 
One of those approaches was proposed by~\cite{lagaris1998artificial} is the method of \emph{residual minimisation}. 
It has received growing attention in the recent advancement of neural network based methods where it is often referred to as \emph{physics informed neural networks} or PINNs following the work of~\cite{raissi2019physics}. 
The idea of residual minimisation is to perceive the Dirichlet problem as the problem of minimising the \emph{residual energy}
    \[ \mathcal{L}\colon \Theta\to\mathbb R, \quad \theta\mapsto \lVert  \Delta u_\theta + f \rVert_{L^2(\Omega)}^2 + \tau \lVert u_\theta - g \rVert_{L^2(\partial\Omega)}^2 \]
for some \(\tau\in(0, \infty)\). 

As \(\mathcal{L}(\theta)\) measures the deviation of \(-\Delta u_\theta\) of the right hand side \(f\) and of the boundary values, it is a natural approach to minimise \(\mathcal{L}\) over the parameter space \(\Theta\). In order to give a theoretical justification for this approach, one can exploit elliptic regularity theory in order to show the a posteriori estimate
\begin{equation}\label{eq:EstGoal}
     \lVert u_\theta - u_f \rVert_{H^s(\Omega)}^2 \le c \cdot \mathcal{L}(\theta). 
\end{equation}
If such an estimate holds, then successful minimisation of \(\mathcal{L}\) implies approximation of the solution \(u_f\) in some norm. Note that an estimate of the form \eqref{eq:EstGoal} provides an a posteriori error estimate which is readily accessible throughout the training process.
In particular, if \(\theta^\ast\in \Theta\) minimises \(\mathcal{L}\) over \(\Theta\), then
\begin{equation}\label{eq:smallest_error_loss} 
    \lVert u_{\theta^\ast} - u_f \rVert_{H^s(\Omega)}^2 
    \le
    c \cdot \inf_{\theta\in \Theta} \mathcal{L}(\theta) 
    \le
    \tilde c \cdot \inf_{\theta\in \Theta} \lVert u_\theta - u_f \rVert_{H^2(\Omega)}^2. 
\end{equation}
In particular, approximation capabilities in \(H^2\) in combination with successful training imply error estimates in \(H^s\). This directly yields that only \(s\le2\) are potentially achievable in the a posteriori estimate~\eqref{eq:EstGoal}.
For \(s\le 1/2\) 
an according estimate has been established by \cite{shin2020error}, i.e., \(H^2\) approximation and successful training implies convergence in \(H^{1/2}\). In this sense, the approach of residual minimisation leads to a loss in regularity of \(3/2\).
It is the purpose of these notes to show that it is not possible to obtain error estimates for \(s>1/2\) in general and that 
for function classes with exact boundary values, i.e. \(u_\theta\in H^2(\Omega)\cap H^1_g(\Omega)\)\footnote{here, \(H^1_g(\Omega)\) denotes the affine subspace of \(H^1(\Omega)\) of functions that agree with \(g\) on \(\partial\Omega\)} the estimates can be improved to hold for all \(s\le2\). Hence, in this case the loss of regularity can be mitigated which shows a theoretical advantage of ansatz classes which exactly satisfy the boundary conditions.  Our proofs rely on curvature based estimates.
We discuss how measuring the residual in Sobolev norms implies estimates in higher order Sobolev norms and generalise our findings to a class of parabolic evolution equations. 

In both the stationary and instationary case our estimates are with respect to stronger norms than for the case of penalised boundary values
, see \cite{shin2020error, mishra2022estimates}. 
More precisely, Theorem~\ref{thm:MainTheorem2} shows that $H^2$ approximation yields estimates in $H^2$ for exact boundary values, where it is only able to provide $H^{1/2}$ estimates for penalised boundary values.
Where we work with the population loss, those works also include the quadrature of the appearing integrals. This routine work can be transferred to our analysis without problems.

\paragraph{Exact Dirichlet Boundary Conditions for Neural Networks}
Standard neural network architectures are usually designed for unconstrained optimization, e.g. it is hardly possible to encode boundary values into a standard neural network architecture. However, following \cite{lagaris1998artificial}, we can transform any unconstrained neural network architecture into an ansatz space with the desired boundary conditions. Assume we want to solve the Poisson problem \eqref{eq:Dir} on $\Omega$ with boundary values $g$. We then construct a smooth function $L\colon\Omega\to[0,\infty)$ that satisfies $L|_{\partial\Omega} = 0$ and $L|_{\Omega}\neq0$. The function $L$ is often referred to as a smooth approximation of the distance function to $\partial\Omega$. Furthermore, we denote by $G\in H^2(\Omega)$ a lift of the boundary conditions to all of $\Omega$, i.e. $G$ satisfies $G|_{\partial\Omega}=g$. For any neural network family $\{u_\theta \mid \theta\in\Theta\} \subseteq H^2(\Omega)$ we then consider the associated family
\begin{equation}\label{eq:ConstructionExactBC}
    \left\{ L\cdot u_\theta + G \mid \theta \in \Theta \right\} \subseteq H^2(\Omega)\cap H^1_g(\Omega)
\end{equation}
and use these functions to approximate the solution of \eqref{eq:Dir}. For complex domains it is difficult to obtain $L$ and $G$ analytically and thus the approximation via neural networks was proposed by \cite{berg2018unified}. For time dependent problems, a similar construction to \eqref{eq:ConstructionExactBC} using a smoothed distance function to the parabolic boundary of the space-time domain can be used. We refer the reader to \cite{lyu2020enforcing} for an explicit example. 

Using ansatz functions with exact boundary values has become increasingly popular as it has been observed to simplify the training process and produce more accurate solutions, see for instance \cite{berg2018unified, roth2021neural, lyu2020enforcing, chen2020comparison}. \mz The works of  \cite{chen2020comparison, courte2021robin} explicitly compare penalized boundary conditions to exactly enforced ones in numerical studies and found improved accuracy and a faster training process. This is in accordance with \cite{krishnapriyan2021characterizing} that illustrates the difficulties in the training process stemming from soft penalties in residual minimisation. \eee It is also possible to encode Neumann or Robin boundary conditions in a similar way, we refer the reader to \cite{lyu2020enforcing}. However, we mention that the approximation capabilities of such ansatz classes have not been studied so far.

\subsection{Organisation}
We state our main results in Section \ref{sec:ImportanceOfExactBV}, where we begin with the improved mode of convergence in the presence of exact boundary conditions in Section \ref{sec:H2Estimates}. We discuss the failure of \(H^2\) estimates with inexact boundary conditions in Section \ref{sec:Failure} and the implication of stronger penalisation of the residual in Section \ref{sec:HigherOrderResidual}. We generalise our results to parabolic equations in Section \ref{sec:Parabolic}.

\subsection{Notation}
On an open subset $\Omega$ of $\mathbb{R}^d$ with boundary $\partial\Omega$, we denote the Lebesgue spaces with integrability order $p\in[1,\infty]$ by $L^p(\Omega)$ and $L^p(\partial\Omega)$ respectively. The Sobolev spaces with order of integrability $p\in[1,\infty]$ and smoothness parameter $s\in\mathbb R$ are denoted by $W^{s,p}(\Omega)$ and $W^{s,p}(\partial\Omega)$. If $p=2$ we instead write $H^s(\Omega)$ and $H^s(\partial\Omega)$. The Sobolev space with zero boundary values in the trace sense is denoted by $W^{1,p}_0(\Omega)$ or $H^1_0(\Omega)$ respectively. The reader is referred to \cite{grisvard2011elliptic} for the precise definitions and more information on Sobolev spaces. Furthermore, for a Banach space $V$, we denote by $L^2(I,V)$ the Bochner space of Bochner measurable, square integrable functions defined on the time interval $I=[0,T]$ taking values in $V$. For the Sobolev space modelled on $L^2(I,V)$ we write $H^1(I,V)$ and refer the reader to \cite{boyer2012mathematical}.

For multivariate, scalar valued functions, we use the symbol $\Delta$ for the Laplace operator, $\nabla$ for the gradient and $D^2$ for the second derivative, i.e. the Hessian. By $d_t$ we denote a time derivative, usually in the context of a vector-valued Sobolev space.

For the purpose of this work, it suffices to view neural networks as parametric families of functions. More precisely, let $\Theta \subseteq \mathbb{R}^N$ be a parameter set and $V$ a (subset of a) function space. Suppose there is a map $\Theta\to V$ denoted by $\theta \mapsto u_\theta$. We then say $u_\theta$ is the (neural network) function represented by the parameters $\theta$. Note that the usual neural network architectures fall in this category, however, this setting also includes general parametric function classes. We stick to the terminology of neural networks since their recent application is the main motivation of this work.

\section{Implications  
of Exact Boundary Values in Residual Minimisation}\label{sec:ImportanceOfExactBV}
In this section we show the theoretical benefits of using neural network type ansatz functions that satisfy Dirichlet boundary conditions exactly in the residual minimisation method for the Poisson problem. We see that the exact boundary conditions improve the mode of convergence from $H^{1/2}$ to $H^2$. Although being formulated for the Laplace operator, those results hold for any elliptic operator which is \(H^2\) regular.

\subsection{\(H^2\) estimates for Residual Minimisation with Exact Boundary Values}\label{sec:H2Estimates}

We start by considering the case of exact boundary conditions and present two main results, one that allows to quantify the $H^2$ error using the value of the loss function and the other, an estimate based on C\'ea's Lemma that allows to link the approximation capabilities of the network class to the error made by residual minimisation.
\begin{setting}\label{setting:laplace}
     We consider again \eqref{eq:Dir}, in particular, we assume that the problem is \(H^2\) regular meaning that there is a constant $C_{\textrm{reg}}>0$, satisfying 
        \[ \lVert u \rVert_{H^2(\Omega)} \leq C_{\textrm{reg}} \lVert \Delta u \rVert_{L^2(\Omega)}\quad \text{for all } u\in H^2(\Omega)\cap H^1_0(\Omega). \]
        Furthermore, we assume that $\Theta$ is a parameter set of a neural network type ansatz class, such that for every $\theta \in \Theta$ we have $u_\theta \in H^2(\Omega)$ and $(u_\theta)|_{\partial\Omega} = g$. As our strategy is to minimise the residual we define the loss function
     \[
        \mathcal{L}\colon\Theta\to\mathbb{R}, \quad \mathcal{L}(\theta) = \lVert \Delta u_\theta + f \rVert^2_{L^2(\Omega)}.
     \]
\end{setting}

\begin{remark}\label{remark:regularity}
    Setting \ref{setting:laplace} is for example satisfied when $\partial\Omega \in C^{1,1}$, $f\in L^2(\Omega)$. 
    Alternatively, one can replace the assumption $\partial\Omega \in C^{1,1}$ by requiring that the domain $\Omega$ is convex. We refer to \cite{grisvard2011elliptic} for a detailed discussion of the regularity properties of elliptic equations.
\end{remark}

The following result is a trivial corollary of the $H^2$ regularity we assumed and a similar result
is due to \cite{van2020optimally}, although not exploiting the benefits of exact boundary conditions. Albeit being of simple nature, we believe it can be of practical relevance due to its easy and explicit error control.
\begin{theorem}\label{thm:MainTheorem1}
    Assume we are in the situation of Setting \ref{setting:laplace}, then it holds for every $\theta\in\Theta$ that 
    \[ \lVert u_\theta - u_f \rVert_{H^2(\Omega)} \leq C_{\textrm{reg}}\sqrt{\mathcal{L}(\theta)}. \]
    For convex domains, we may estimate the regularity constant explicitely. It holds
    \begin{equation*}
        C_{\textrm{reg}} \leq \sqrt{1 + C_P} \leq \sqrt{ 1 + \left( \frac{|\Omega|}{\omega_d} \right)^{\frac1d} },
    \end{equation*}
    where $d$ is the dimension of $\Omega$, $\omega_d$ denotes the volume of the unit ball in $\mathbb{R}^d$ and $C_P$ is the Poincar\'e constant for functions in $H^1_0(\Omega)$.
\end{theorem}
\begin{proof}
    The difference $u_\theta - u_f$ lies in $H^2(\Omega)\cap H^1_0(\Omega)$ and solves $-\Delta(u_f-u_\theta) = \Delta u_\theta +f$. The $H^2(\Omega)$ regularity theory then implies the desired estimate. Let us now assume that $\Omega$ is convex and derive the explicit estimate on $C_{\textrm{reg}}$. We expand the $H^2(\Omega)$ norm of $u_f \in H^2(\Omega)\cap H^1_0(\Omega)$
    \[ \lVert u_f \rVert^2_{H^2(\Omega)} = \lVert u_f \rVert^2_{L^2(\Omega)} + \lVert \nabla u_f \rVert^2_{L^2(\Omega)} + \lVert D^2u_f \rVert^2_{L^2(\Omega)} \]
    Due to the zero boundary values and the convexity of $\Omega$ we have
    \[ \lVert D^2u \rVert^2_{L^2(\Omega)} = \lVert \Delta u \rVert^2_{L^2(\Omega)} = \lVert f \rVert^2_{L^2(\Omega)} \]
    and we refer the reader to \cite{grisvard2011elliptic} for details. The first two terms can be estimated jointly using the a priori estimates of the Lax-Milgram Theorem, this yields
    \[ \lVert u \rVert^2_{L^2(\Omega)} + \lVert \nabla u \rVert^2_{L^2(\Omega)} \leq C_P^2\lVert f \rVert^2_{H^1_0(\Omega)^*} \leq C_P^2\lVert f \rVert^2_{L^2(\Omega)}. \]
    This is due to the fact that $C_P^{-1}$ is the coercivity constant of the Dirichlet Laplacian bilinear form, see \cite{evans1998partial}. The explicit estimate of the Poincar\'e constant $C_P$ can be found in \cite{jostpartial}.
\end{proof}
\begin{remark}
Some remarks are in order.
\begin{itemize}
    \item [(i)] The zero boundary conditions are essential. If one instead resorts to a $L^2(\partial\Omega)$ penalty of the boundary values the best convergence one can hope for is $H^{1/2}(\Omega)$. We elaborate this in Section \ref{sec:Failure}.
    \item[(ii)] The theorem allows to compute an explicit upper bound on the error made by residual minimisation, once the training returns a parameter \(\theta\) via computing the (continuous) loss. In particular, no access to the solution $u_f$ is required. This means that if boundary conditions are encoded in the ansatz functions, the loss itself is a consistent a posteriori error estimator for the residual minimisation method.
    \item[(iii)] The root in the estimate above does not indicate a slow convergence. In fact, the loss itself is a squared $L^2(\Omega)$ norm and the root accounts for that.
\end{itemize}
\end{remark}
The next theorem allows to quantify the error made by the residual minimisation method using the optimization quality and the expressiveness of the ansatz class. It is an application of the non-linear C\'ea Lemma as formulated by \cite{muller2021error}.
\begin{theorem}\label{thm:MainTheorem2}
Assume we are in Setting \ref{setting:laplace}, then for any $\theta\in\Theta$ it holds
\begin{equation}
    \lVert u_f - u_\theta \rVert_{H^2(\Omega)} \leq \sqrt{ C^2_{\textrm{reg}}\delta + C^2_{\textrm{reg}}\inf_{\tilde{\theta}\in\Theta }\lVert \Delta(u_{\tilde \theta} - u_f) \rVert^2_{L^2(\Omega)} } \le \sqrt{ C^2_{\textrm{reg}}\delta + C^2_{\textrm{reg}}\inf_{\tilde{\theta}\in\Theta }\lVert u_{\tilde \theta} - u_f \rVert^2_{H^2(\Omega)} },
\end{equation}
where $\delta = \mathcal{L}(\theta) - \inf_{\tilde \theta \in \Theta}\mathcal{L}(\tilde\theta)$.
\end{theorem}
\begin{proof}
We define the energy
\[ E\colon H^2(\Omega) \to \mathbb{R},\quad E(u) = \lVert \Delta u + f \rVert^2_{L^2(\Omega)}.\]
Note that $E$ is defined on a different domain than the loss function $\mathcal{L}$, which is why we reserve an own symbol for it. The energy $E$ is a quadratic energy
\begin{align*} 
    \lVert \Delta u + f \rVert^2_{L^2(\Omega)} &= \int_\Omega (\Delta u)^2\mathrm dx + 2\int_\Omega f \Delta u  \mathrm dx + \int_\Omega f^2 \mathrm dx
    \\
    &= \frac{1}{2}a(u,u) - F(u) + c,
\end{align*}
where the bilinear form $a:H^2(\Omega)\times H^2(\Omega)\to\mathbb{R}$, the functional $F\in H^2(\Omega)^*$ and the constant $c$ are given by
\begin{gather*}
a(u,v) = 2\int_\Omega \Delta u \Delta v \mathrm dx, \quad F(u) = 2\int_\Omega f (-\Delta u)  \mathrm dx, \quad c = \int_\Omega f^2 \mathrm dx.
\end{gather*}
The unique minimiser of $E$ in the affine subspace $H^2(\Omega)\cap H^1_g(\Omega)$ is precisely the solution $u_f$ to the Poisson problem \eqref{eq:Dir}. The bilinear form $a$ is coercive on the subspace $H^2(\Omega)\cap H^1_0(\Omega)$, which follows from elliptic regularity theory, see for instance \cite{grisvard2011elliptic}. This allows to exploit a C\'ea Lemma for non-linear ansatz spaces, as described in Proposition 3.1 by \cite{muller2021error}. To transfer this to the affine space $H^2(\Omega)\cap H^1_g(\Omega)$ we choose $u_g \in H^2(\Omega)$ such that $-\Delta u_g = 0$ and $(u_g)|_{\partial\Omega} = g$. For an arbitrary $u_\theta$ we then expand
\begin{equation*}
    \lVert u_\theta - u_f \rVert_{H^2(\Omega)} = \lVert (u_\theta - u_g) - (u_f - u_g) \rVert_{H^2(\Omega)}. 
\end{equation*}
Now note that $u_f - u_g$ solves $-\Delta (u_f - u_g) = f$ with zero boundary values, hence $u_f - u_g$ is the unique minimiser of $E$ over the subspace $H^2(\Omega)\cap H^1_0(\Omega)$ and we can apply C\'ea's Lemma with the ansatz set $\{u_\theta - u_g\mid \theta \in \Theta \}$
\begin{align*}
    \lVert u_\theta - u_f \rVert_{H^2(\Omega)} \leq \sqrt{\frac{2\delta}{\alpha} +\frac1\alpha \inf_{\tilde{\theta}\in\Theta } \lVert u_{\tilde\theta} - u_f\rVert^2_a  },
\end{align*}
where \(\lVert \cdot\rVert_a\) denotes the norm induced by \(a\). Using that the coercivity constant $\alpha$ of $a$ is $2/C^2_{\textrm{reg}}$ and the norm $\lVert \cdot \rVert_a = 2\lVert \Delta \cdot\rVert_{L^2(\Omega)}$ we conclude.
\end{proof}

\begin{remark}[General Elliptic Equations]\label{remark:general_elliptic_equations}
    The discussion of this chapter can be extended to more general elliptic equations. For coefficients $A \in C^{0,1}(\Omega, \mathbb{R}^{d\times d})$, a right-hand side $f\in L^2(\Omega)$ and boundary values $g\in H^{3/2}(\partial\Omega)$ consider the equation
    \begin{equation*}
    \begin{split}
        -\operatorname{div}\left(A \nabla u \right)  & = f \quad \text{in } \Omega, \\
    u & = g \quad \text{on } \partial\Omega.
    \end{split}
\end{equation*}
    If we assume that $\partial\Omega \in C^{1,1}$ (or that $\Omega$ is convex) and the coefficients are uniformly elliptic, i.e., for a constant $c_A >0$ satisfy $A(x)\xi\cdot\xi \geq c_A|\xi|^2$ uniformly in $x\in\Omega$ and $\xi\in\mathbb{R}^d$, the problem admits a unique solution $u_f \in H^{2}(\Omega)$ and we can estimate
    \begin{equation*}
        \lVert u_f \rVert_{H^2(\Omega)} \leq c_{\text{reg}} \left( \lVert f \rVert_{L^2(\Omega)} + \lVert g \rVert_{H^{3/2}(\partial\Omega)} \right).
    \end{equation*}
    Arguing as in the proof of Theorem~\ref{thm:MainTheorem1} we obtain 
    \begin{equation*}
        \lVert u_\theta - u_f \rVert_{H^2(\Omega)} \leq c_{\text{reg}} \sqrt{\mathcal{L}(\theta)}.
    \end{equation*}
    Similarly, Theorem~\ref{thm:MainTheorem2} can be transferred to this setting.
\end{remark}

\subsection{Failure without Exact Boundary Values}\label{sec:Failure}
In this section we show that \emph{not} enforcing exact boundary values in the neural network ansatz functions leads to considerably weaker error estimates. Throughout this subsection, we work under the following assumptions.
\begin{setting}\label{set:failure}
     We consider again \eqref{eq:Dir}.
        We assume that $\Theta$ is a parameter set of a neural network type ansatz class, such that for every $\theta \in \Theta$ we have $u_\theta \in H^2(\Omega)$, but make no assumptions on its boundary values. As our strategy is to minimise the residual we define the loss function with boundary penalty
     \[
        \mathcal{L}_\tau\colon\Theta\to\mathbb{R}, \quad \mathcal{L}_\tau(\theta) = \lVert \Delta u_\theta + f \rVert^2_{L^2(\Omega)} + \tau \lVert u_\theta - g \rVert_{L^2(\partial\Omega)}^2,
     \]
     where $\tau\in(0, \infty)$ is a positive penalization parameter.
\end{setting}

\jm{Without exact boundary values, the penalization of the deviations of the boundary values is required in order to enforce them approximately. Note that if $u_\theta$ has exact boundary values, it holds that $\mathcal L_\tau(\theta) = \mathcal L(\theta)$. With the penalization introduced above, we obtain a similar result to Theorem~\ref{thm:MainTheorem1} but only with respect to the weaker $H^{1/2}$-norm, which is to the estimate by~\cite{shin2020error} is sharp. However, we sharpen this result by showing that $1/2$ is the largest exponent for which such an estimate can hold in general.
}

\begin{theorem}
Assume that we are in Setting~\ref{set:failure} and that the domain $\Omega\subseteq\mathbb R^d$ has a smooth boundary $\partial\Omega\in C^\infty$. Then for $s\in\mathbb R$ there is a constant $c>0$ such that 
\begin{equation}\label{eq:abstractEstimate}
    \lVert u_\theta - u_f \rVert_{H^s(\Omega)} \le c \sqrt{\mathcal L_\tau(\theta)} \quad \text{for all } \theta\in\Theta
\end{equation}
and all parametric classes and data $f\in L^2(\Omega), g\in H^{3/2}(\partial\Omega)$ if and only if $s\le 1/2$.
\end{theorem}
\begin{proof}
\jm{First, we show that the estimate holds for $s\le1/2$, where it suffices to show it for $s=1/2$. For this, we use the estimate 
\begin{equation}\label{eq:RegularityEstimate}
    \lVert u \rVert_{H^s(\Omega)} \leq c \left( \lVert -\Delta u \rVert_{H^{s-2}(\Omega)} + \lVert u \rVert_{H^{s-1/2}(\partial\Omega)}\right),
\end{equation}
for all $u\in C^\infty(\overline\Omega)$ and $s\in\mathbb R$, 
see 
Theorem 2.1 in~\cite{schechter1963p} or Lemma 6.2 in~\cite{shin2020error}. 
Setting $s=1/2$ and noting that it extends to functions $u\in H^2(\Omega)$ yields
\[
    \lVert u \rVert_{H^{1/2}(\Omega)}
    \leq
    c\left( \lVert \Delta u \rVert_{H^{-3/2}(\Omega)} + \lVert u \rVert_{L^{2}(\partial\Omega)} \right) 
    \leq
    \tilde c\left( \lVert \Delta u \rVert_{L^{2}(\Omega)} + \lVert u \rVert_{L^{2}(\partial\Omega)} \right).
\]
Setting $u\coloneqq u_\theta - u_f$ yields
\[ \lVert u_\theta - u_f \rVert_{H^{1/2}(\Omega)} \leq (1 + \tau^{-1/2}) \tilde c \sqrt{\mathcal{L}_\tau(\theta)}. \]
}

\jm{To show that the estimate~\eqref{eq:abstractEstimate} can not in general be established for any stronger norms, we assume that it holds for some $s\in\mathbb R$. As in the proof of Theorem \ref{thm:MainTheorem2} we define the energy, this time penalising boundary values
\[ E_\tau\colon H^2(\Omega) \to \mathbb{R},\quad E_\tau(u) \coloneqq \lVert \Delta u + f \rVert^2_{L^2(\Omega)} + \tau \lVert u - g  \rVert_{L^2(\partial\Omega)}^2. \]
If the estimate~\eqref{eq:abstractEstimate} holds for general parametric classes, this yields
\[
\lVert v - u_f \rVert_{H^s(\Omega)}^2 \leq c \cdot E_\tau(v) \quad \text{for all } v\in H^2(\Omega), f\in L^2(\Omega), g\in H^{3/2}(\partial\Omega).
\]
Choosing \(f=0\) and \(g=0\) yields 
    \[ \lVert v \rVert_{H^s(\Omega)}^2 \le c\cdot E_\tau(v) = c \cdot \left( \lVert \Delta v \rVert_{L^2(\Omega)}^2 + \tau \lVert v \rVert_{L^2(\partial\Omega)}^2\right) \quad \text{for all } v\in H^2(\Omega). \]
For \(h\in H^{3/2}(\partial\Omega)\) let \(u_h\in H^2(\Omega)\) denote the unique harmonic extension, i.e., the solution of 
\begin{equation*}
    \begin{split}
        -\Delta u_h & = 0 \quad \text{in } \Omega \\
        u_h & = h \quad \text{on } \partial\Omega.
    \end{split}
\end{equation*}
Now we have
\begin{equation}\label{eq:help}
    \lVert h \rVert_{H^{s-1/2}(\partial\Omega)}^2 \le c \rVert u_h \rVert_{H^s(\Omega)}^2 \le \tilde c\left( \lVert \Delta u_h \rVert_{L^2(\Omega)}^2 + \tau \lVert u_h \rVert_{L^2(\partial\Omega)}^2\right) = \tilde c \tau \cdot \lVert h \rVert_{L^2(\partial\Omega)}^2
\end{equation}
for all $h\in H^{3/2}(\partial\Omega)$.
In order to see that this implies $s\le1/2$ we assume the contraty and set $\varepsilon\coloneqq s-1/2>0$. Then, the embedding $H^{3/2}(\partial\Omega) \hookrightarrow H^\varepsilon(\partial\Omega)$ is dense and hence~\eqref{eq:help} extends to $h\in H^\varepsilon(\partial\Omega)$. 
This yields that all norms $\lVert\cdot\rVert_{H^\delta(\partial\Omega)}$ for $\delta\in(0, \varepsilon)$ are equivalent to $\lVert\cdot\rVert_{L^2(\partial\Omega)}$, which implies that all spaces $H^{\delta}(\partial\Omega)$ agree which constitutes a contradiction.}
\end{proof}

\begin{remark}[Stronger estimates through stronger penalty]
We have seen that the $L^2(\partial\Omega)$ penalisation can not lead to estimates in a stronger Sobolev norm than $H^{1/2}(\Omega)$. However, inspecting inequality \eqref{eq:RegularityEstimate} one could -- at least in theory -- penalise the boundary values in the $H^{3/2}(\partial\Omega)$ norm and would then obtain $H^2(\Omega)$ estimates. As the $H^{3/2}(\partial\Omega)$ norm is difficult to approximate in practice, this is no feasible numerical approach.
\end{remark}
\begin{remark}[Stronger estimates through interpolation]
It is possible to bound the \(H^s\) error for \(s\ge1/2\) of residual minimisation with \(L^2\) boundary penalty for the expense of worse rates and under the cost of an additional factor for which it is not clear whether it is bounded.
Similar to \cite{biswas2020error} one can use an interpolation inequality for \(s\in[1/2, 2]\) to obtain
    \[ \lVert u \rVert_{H^s(\Omega)} \le \lVert u \rVert_{H^{1/2}(\Omega)}^{2(2-s)/3} \cdot \lVert  u \rVert_{H^2(\Omega)}^{(2s-1)/3} \quad \text{for all } u\in H^2(\Omega). \]
Together with the a posteriori estimate on the \(H^{1/2}\) norm, this yields 
\begin{align*}
    \lVert u_f-u_\theta \rVert_{H^s(\Omega)} & \le \lVert u_f-u_\theta \rVert_{H^{1/2}(\Omega)}^{2(2-s)/3} \cdot \lVert u_f-u_\theta \rVert_{H^2(\Omega)}^{(2s-1)/3} \precsim \lVert u_f-u_\theta \rVert_{H^2(\Omega)}^{(2s-1)/3} \cdot L(\theta)
    ^{(2-s)/3} \\ &\le \left(\lVert u_f \rVert_{H^2(\Omega)} + \lVert u_\theta \rVert_{H^2(\Omega)}\right)^{(2s-1)/3} \cdot L(\theta)^{(2-s)/3}.
\end{align*}
Hence, if it is possible to control the \(H^2\) norm of the neural network functions, one obtains an a posteriori estimate on the \(H^s\) error. Note however, that the \(H^2\) norm of the neural networks functions is not controlled through the loss function \(L\) and hence, this estimates requires an additional explicit or implicit control on the \(H^2\) norm in order to be informative.
Note, however, 
that the power of the a posteriori estimate decreases towards zero for \(s\to2\) and the estimate collapses to a trivial bound for \(s=2\).
\end{remark}

\subsection{Higher Order Sobolev Norms as a Residual Measurement}\label{sec:HigherOrderResidual}
We discuss the potential benefit of using (higher order) Sobolev norms to measure the residual, as was already proposed by \cite{son2021sobolev}. We are again supposing the exact enforcement of boundary conditions. Our precise setting is the following.

\begin{setting}\label{setting:SobolevTraining}
Let $p\in(1,\infty)$ and $k\geq0$ be fixed. Assume that $\Omega\subseteq\mathbb{R}^d$ is a bounded, open domain with $C^{k+1,1}$ boundary and let $f\in W^{k,p}(\Omega)$ and $g\in W^{2+k-1/p,p}(\partial\Omega)$. Denote by $u_f$ the solution to \eqref{eq:Dir}. Furthermore, let $\Theta$ be a parameter set of a neural network class, such that for every $\theta \in \Theta$ we have $u_\theta \in W^{k+2,p}(\Omega)$ and $u|_{\partial\Omega} = g$. We define the loss function
\begin{equation}\label{eq:SobolevLossFunction}
    \mathcal{L}\colon\Theta\to \mathbb{R}, \quad \mathcal{L}(\theta) = \lVert \Delta u_\theta + f\rVert^p_{W^{k,p}(\Omega)}.
\end{equation}
\end{setting}

In total analogy to Theorem \ref{thm:MainTheorem1} we obtain the following result.
\begin{theorem}
Assume we are in the situation of Setting \ref{setting:SobolevTraining}, then it holds for every $\theta\in\Theta$ that 
\[
    \lVert u_\theta - u_f \rVert_{W^{k+2,p}(\Omega)} \leq C_{\textrm{reg}}(p,k)\sqrt[p]{\mathcal{L}(\theta)}.
\]
\end{theorem}
\begin{proof}
The essential ingredient is the $L^p$ regularity theory that holds under the assumptions made in Setting \ref{setting:SobolevTraining}, see for instance chapter 2.5 in \cite{grisvard2011elliptic}. The relevant result is that
\[
-\Delta\colon W^{k+2,p}(\Omega)\cap W^{1,p}_0(\Omega)\to W^{k,p}(\Omega)
\]
is a linear homeomorphism, where $C_{\textrm{reg}}(p,k)$ denotes the operator norm of its inverse.
\end{proof}
\begin{remark}
The above result might be interesting if approximation of higher derivatives is desired. Furthermore, the empirical findings of \cite{son2021sobolev} suggest that measuring the residual in a Sobolev norm might lead to fewer iterations in a gradient based optimization routine. 
\end{remark}

\section{Estimates for Parabolic Equations}\label{sec:Parabolic}
The same observation made for the Poisson equation can be exploited for linear parabolic equations when both initial and boundary values are satisfied exactly by the ansatz class. Here, the key is maximal parabolic $L^2$ regularity theory. We begin by describing our setting.

\begin{setting}\label{setting:parabolic}
We consider again a domain $\Omega\subseteq\mathbb{R}^d$ that is $H^2$ regular for the Laplacian and a finite time interval $I = [0,T]$. For $f\in L^2(I,L^2(\Omega))$, $g\in H^{3/2}(\partial\Omega)$ and $u_0 \in H^1_0(\Omega)$ we consider the parabolic problem
\begin{equation}
    \begin{split}
        d_tu - \Delta u & = f \quad \text{in } I\times\Omega \\
    u(t)|_{\partial\Omega} &= g \quad \text{for all }t\in I \\
    u(0) & = u_0.
    \end{split}
\end{equation}
Let $\Theta$ be a parameter set of a neural network class such that for every $\theta\in\Theta$ the function $u_\theta$ is a member of the space 
\[
    \mathcal{X} = H^1(I,L^2(\Omega))\cap L^2(I,H^2(\Omega)\cap H^1_g(\Omega)), \mz \quad \lVert u \rVert_{\mathcal{X}} = \lVert d_tu \rVert_{L^2(I,L^2(\Omega))} + \lVert u \rVert_{L^2(I,H^2(\Omega)) \eee}
\] 
with $u_\theta(0) = u_0$. This means that both initial and boundary conditions are satisfied exactly. For an introduction to vector-valued Sobolev spaces we refer the reader to \cite{boyer2012mathematical}. Then we define the loss function
\[\mathcal{L}(\theta) = \lVert d_tu_\theta -\Delta u_\theta - f \rVert^2_{L^2(I,L^2(\Omega))}\]
\end{setting}

The following theorem is analogue to the case of the Laplacian and relies on a parabolic regularity result.

\begin{theorem}\label{thm:parabolic}
Assume xwe are in Setting \ref{setting:parabolic}. Then it holds for all $\theta \in \Theta$ that 
\[ \lVert  u_\theta - u_f \rVert_{\mathcal{X}} \leq C\sqrt{\mathcal{L}(\theta)} \]
\end{theorem}
\begin{proof}
    We denote by $H^1_0(I,L^2(\Omega))$ the vector-valued Sobolev space with vanishing initial values.
    Maximal parabolic $L^2(\Omega)$ regularity theory tells us that 
    \begin{equation*}
        d_t - \Delta\colon H_0^1(I,L^2(\Omega))\cap L^2(I,H^2(\Omega)\cap H^1_0(\Omega)) \longrightarrow L^2(I,L^2(\Omega))
    \end{equation*}
    is a linear homeomorphism and this implies the assertion, see for instance \cite{arendt2017jl} for more information on maximal parabolic regularity. The constant $C$ is then the operatornorm of $(d_t - \Delta)^{-1}$.
\end{proof}
\begin{remark}
Of course this result is not limited to the heat equation. Indeed one can replace $-\Delta$ by a self-adjoint, coercive operator that satisfies $H^2(\Omega)$ regularity, we refer the reader again to \cite{arendt2017jl} for the corresponding regularity theory. \mz For information on the dependency of the constant $C$ on data, we refer to \cite{amann1995linear}, especially Theorem 4.10.8. \eee
\end{remark}
\begin{remark}
 \cite{mishra2022estimates} report error estimates for parabolic equations not enforcing initial and boundary conditions in the ansatz architecture. We stress that even though the solutions there are assumed to be classical, smooth solutions the error is only estimated in the $L^2(I\times\Omega)$ norm which is weaker than the estimates presented here. This is again due to advantage of exact boundary and initial conditions.
\end{remark}

\acks{The authors want to thank Luca Courte, Patrick Dondl and Stephan Wojtowytsch for their valuable comments. 
JM acknowledges support by the Evangelisches Studienwerk e.V. (Villigst), the International Max Planck Research School for Mathematics in the Sciences (IMPRS MiS) and the European Research Council (ERC) under the EuropeanUnion’s Horizon 2020 research and innovation programme (grant number 757983). MZ acknowledges support from BMBF within the e:Med program in the SyMBoD consortium (grant number 01ZX1910C) and the Research Council of Norway (grant number 303362).}

\bibliography{references}

\end{document}